\documentclass[12pt]{amsart}
\usepackage{hyperref}
\usepackage{amsmath,amsthm,amsfonts,amssymb,amscd}
\theoremstyle{plain}
\usepackage[svgnames]{xcolor}
\usepackage{mathrsfs}
\usepackage{eepic}
\usepackage[centering]{geometry}

\newtheorem{theorem}{\bf Theorem}[section]
\newtheorem{lem}[theorem]{\bf Lemma}
\newtheorem{defi}[theorem]{\bf Definition}
\newtheorem{rem}[theorem]{\bf Remark}
\newtheorem{cor}[theorem]{\bf Corollary}
\newtheorem{ex}[theorem]{\bf Example}
\newtheorem{conv}[theorem]{\bf Convention}
\newtheorem*{conj*}{\bf Conjecture}

\def\rank{\ensuremath{\mathrm{rank}}}

\def\diam{\ensuremath{\mathrm{diam}}}

\newtheorem{notation}[theorem]{\bf Notation}

\begin{document}
\title{Upper bounds for the diameter of a direct power of non-abelian solvable groups}
\author{ Azizollah Azad}
\address {Department of Mathematics, Faculty of Sciences Arak University, Arak, Iran.}
\email{a-azad@araku.ac.ir}
\author{ Nasim Karimi}
\address{Instituto de Matem\'atica e estat\'istica, Universidade do Estado de  Rio de janeiro, Rio de Janeiro,Brasil}
\email{nasim@ime.uerj.br}
\date{\today}
\maketitle

\begin{abstract} Let $G$ be a finite group with a generating set $A$. By the (symmetric) diameter of $G$ with respect to $A$ we mean the maximum over $g \in G$ of the length of the shortest word in $(A\cup A^{-1}) A$ expressing $g$.By the (symmetric) diameter of $G$ we mean the maximum of (symmetric) diameter over all generating sets of $G$.  Let $n\geq1$, by $G^n$ we mean the $n$-th direct power of $G$. 
For $n\geq 1$ and finite non-abelian solvable group $G$ we find an upper bound, growing polynomially with respect to $n$, for the symmetric diameter and the diameter of $G^n$.  
 \end{abstract}

\section{Introduction}

 Let $G$ be a finite group with a generating set $A$. By the (symmetric) diameter of $G$ with respect to $A$ we mean the maximum over $g \in G$ of the length of the shortest word in $(A\cup A^{-1})  A$ expressing $g$. Producing a bound for the (symmetric) diameter of a finite group is an important area of research in finite group theory. It is worth to mention to the  Babai's conjecture \cite{Babai&Seress:1988}: every non-abelian finite simple group $G$ has diameter $ \leq log^ k |G|$, where $k$ is an absolute constant; the conjecture is still open, despite great progress towards a solution both for alternating groups and for groups of Lie type. 
 Asymptotic estimate of the symmetric diameters of non-abelian simple groups with respect to various types of generating sets can be find in the survey \cite{Babai&Hetyei&Kantor&Lubotsky&Seress:1990}, which also lists related work, e.g., on the diameters of permutation groups. Furthermore, the area has progressed a lot over the last few years (see for instance \cite{Seress&Helfgott:2014, Seress&Helfgott&Bamberg:2014,Seress&Helfgott&Zuk:2015}). A more modest question is that of finding bounds for the diameter of direct products of finite groups, depending on the diameter of their factors. In this direction the papers \cite{Helfgott:2019,Dona:2022} in which produced upper bounds for the diameter of direct product of non-abelian simple groups are significant.    
 
 Let $n\geq1$, by $G^n$ we mean the $n$-th direct power of $G$. In \cite{Karimi:2017} have been appeared the following question: 
 How large can be the diameter of $G^n$ with respect to any generating set? There have been proved that if $G$ is abelian, then  the diameter of $G^n$ with respect to any generating set is  $O(n)$ and if $G$ is nilpotent, symmetric or dihedral, then there exist a generating set of minimum size which the diameter of $G^n$ with respect to this generating set is $O(n)$. In \cite{Dona:2022} have been proved that if $G$ is a non-abelian simple group, then the diameter of $G^n$ with respect to any generating set is $O(n^3)$.  
 
 Our main goal here is to present upper bounds for the diameter and the symmetric diameter of $G^n$, in which $G$ is a non-abelian solvable group. In fact, we prove that if $G$ is a non-abelian solvable group, then 
 \begin{equation*}
 \label{1}D^s(G^n) \leq \frac{1}{4} (4n )^{l} |G|
 \end{equation*}
 and
 \begin{equation*}
 \label{2}D(G^n) \leq n^{l} |G| \prod_{i=0}^{l-2}( |G^{(i)}| +1),
\end{equation*}
in which $l$ is the length of derived series of $G$.

\section{Preliminaries}

Throughout the paper all groups are considered to be finite . The subset $A\subseteq G$ is a generating set of $G$, if every element of $G$ can be expressed as a sequence of elements in $A$.\footnote{Usually $A\subseteq G$ is considered to be a generating set, if every element of $G$ can be expressed as a sequence of elements in $A\cup A^{-1}$. When $G$ is finite the definitions coincide.}  By the rank of $G$, denoted by $\rank(G)$,  we mean the cardinality of any of the smallest generating sets of $G$. 
 By the length of a non identity element $g\in G$, with respect to $A$, 
 we mean the minimum length of a sequence expressing $g$ in terms of elements in $A$.
 Denote this parameter by $l_A{g}$. Similarly we define the symmetric length of  a non identity element $g\in G$, with respect to $A$,  to be the minimum length of a sequence expressing $g$ in terms of elements in $A \cup A^{-1}$. Denote this parameter by $l^{s}_A{g}$.
 
 \begin{conv}
 We consider the (symmetric) length of identity to be zero, i.e. $l_A(1)=l^{s}_A(1)=0$ for every generating set $A$.  
 \end{conv}
 
 \begin{defi}
 Let $G$ be a finite group with generating set $A$. By the  \emph{diameter} of $G$ with respect to $A$ we mean  
  $$\diam(G,A):=\max \{l_A(g): g \in G \}.$$ 
 And by the symmetric \emph{diameter} of $G$ with respect to $A$ we mean
  $$\diam^{s}(G,A):=\max \{l^{s}_A(g): g \in G \}.$$

 \end{defi} 

 \begin{notation}
 Let $G$ be a finite group with a generating set $A$. Let $S$ be a subset of $G$. Denote by $Ml_A(S)$ the maximum of $l_A(s)$ over all $s$ in $S$ . Note that $Ml_A(G)=\diam(G,A).$
 \end{notation}
 
\begin{notation}
Denote by ($D^s(G)$) $D(G)$ the maximum (symmetric) diameter over all generating sets of $G$.
\end{notation} 

The following theorem has been proved by Wiegold in \cite{Wiegold:1975}.
\begin{theorem}\label{rank-solvable-group}\cite{Wiegold:1975}
Let G be a finite non-trivial solvable group, and set $$\rank(G)=\alpha,~\rank(G/G')=\beta.$$ Then $$\rank(G^n)=\beta n,$$ for $n \geq \alpha / \beta$. 
\end{theorem}

\begin{lem}\label{rank-inequality} {\rm \cite{Wiegold:1974}}
Let $G$ be a finite group and $k$ be a positive integer. Then
\begin{equation}
k\, \rank (G/G') \leq \rank (G^k) \leq k\, \rank(G),
\end{equation}
where $G'$ is the commutator subgroup of $G$.
\end{lem}
The following is an application of Lemma \ref{rank-inequality}.
\begin{cor}\label{property}
Let $G$ be a finite group. If $\rank(G)=\rank(G/G')$, then the following equality holds:
\begin{equation}\label{group-property}
\rank(G^n)=n\, \rank(G).
\end{equation}
  In particular, nilpotent groups satisfy this property.
\end{cor}
\begin{proof}
The first statement is an immediate consequence of Lemma \ref{rank-inequality}. We prove the second statement. Note that, if $H$ is a homomorphic image of a finite group $G$, then $\rank(H) \leq \rank (G)$. Therefore, it is enough to show that $\rank(G) \leq \rank(G/G')$ for every finite nilpotent group $G$. Let $A= \{g_1G',g_2G',\ldots,g_kG' \}$ be a generating set of $G/G'$ of minimum size. Consider an arbitrary element $g \in G$. There exist some $i_1,i_2,\ldots, i_l \in \{1,2,\ldots,k\}$ such that $gG'= g_{i_1}g_{i_2}\ldots g_{i_l}G'$. This shows that $G$ is generated by $\{g_1,g_2,\ldots ,g_k\}$ together with some elements in $G'$. Because $G$ is nilpotent, it is generated by $\{g_1,g_2,\ldots ,g_k\}$ alone, see \cite[page 350]{Magnus&Karrass&Solitar:1976}. Therefore, $\rank(G) \leq \rank(G/G')$, which completes the proof.
\end{proof}

\section{Main results}
We start by presenting an upper bound for the symmetric diameter of a direct power of a non-abelian solvable group. Let $G$ be a non-abelian solvable group. Since solvable groups have a derived series of finite length our strategy is to find a relation between the diameter of a solvable group and the diameter of its derived subgroup. For this we need to establish a relation between the generating sets of the group and the generating sets of its subgroups. The following lemma, well known as \emph{Schreier Lemma},\index{Schreier!Lemma} gives a generating set for a subgroup of a group with respect to a generating set of the whole group. The generators of the subgroup are usually called Schreier generators\index{Schreier! generators}. Using Schreier generators we derive a relation between the diameter of a group and the diameter of its subgroup.

\begin{defi}
Let $H$ be a subgroup of a group $G$. By a right transversal for $G$ mod $H$, we mean a subset of $G$ which intersects every right coset $Hg$ in exactly one element.
\end{defi}

\begin{rem}
Let $G$ be a finite group with a generating set $X$ and a normal subgroup $H$. It is easy to see that the set $HX=\{Hx :~ x \in X\}$ is a generating set of $G/H$. Given an arbitrary element $Hg \in G/H$, $Hg$ can be written as a product of at most $D(G/H)$ elements in $XH$. Hence, there exist $x_1,x_2,\ldots,x_{D(G/H)} \in X$ such that $gH=x_1H x_2 H \ldots x_{D(G/H)} H=x_1x_2\ldots x_{D(G/H)} H$. It shows that there always exists a right transversal $T$ for $G$ mod $H$ such that $$Ml_{X}(T) \leq D(G/H),~ 1 \in T.$$
\end{rem}

\begin{lem} {\rm \cite{Seress:2003}}\label{Schreier}
Let $H \leq G=\langle X \rangle $ and let $T$ be a right transversal for $G$ mod $H$, with $1 \in T.$ Then the set 
$$ \{ t x t_1^{-1} \mid t,t_1  \in T , x \in X, t x t_1^{-1} \in H \}$$ generates $H$.   
\end{lem}

Using Schreie's Lemma leads to the following Lemma which is \cite[Lemma 5.1]{Babai&Seress:1992}.
\begin{lem}\label{break-diameter}
If $1 \not = N $, $N \triangleleft G$, then the following inequalities hold:
$$ D^s(G) \leq 2 D^s(G/N) \, D^s(N) + D^s(G/N) +  D^s(N) \leq 4D^s(G/N) \, D^s(N).$$ 
\end{lem}
Now we are ready to prove the first main theorem. 
\begin{theorem}\label{solvable-symmetric}
If $G$ is a non-abelian solvable group then 
$$D^s(G^n)\frac{1}{4} (4n )^{l} |G|,$$
where $l$ is the length of the derived series of $G$.
\end{theorem}

\begin{proof}
Let  
$$\{1\}=G^{(l)} \triangleleft G^{(l-1)}\triangleleft \cdots \triangleleft G'' \triangleleft G'\triangleleft G$$
be the derived series of the group $G$. Since for $1 \leq i \leq l $ we have $$(G^{(i)})^n = (G^n)^{(i)},$$ the series 
$$\{1\}=(G^{(l)})^n \triangleleft (G^{(l-1)})^n \triangleleft \cdots \triangleleft (G'')^n \triangleleft (G')^n \triangleleft G^n$$
is the derived series of the group $G^n$. Using the second inequality in Lemma \ref{break-diameter}, the maximum of the diameter of the group $G^n$ is bounded above by 
\begin{align}\label{solvable1}
4 ^{l-1} \, D^s(G^n / (G')^n) \, D^s( (G')^n/ (G'')^n) \cdots D^s((G^{(l-2)})^n / (G^{(l-1)})^n) \, D^s((G^{(l-1)})^n).
\end{align}
Whereas, for $1 \leq i \leq l-1$ we have
$$
(G^{(i)})^n / (G^{(i+1)})^n  \cong (G^{(i)}/ G^{(i+1)})^n$$ and the factors in a derived series are abelian, by \cite[theorem 3.2]{Karimi:2017} we get 
\begin{equation}\label{solvable2}
D^s(G^{(i)})^n / (G^{(i+1)})^n \leq n \, |G^{(i)}/G^{(i+1)}|= n  \,|G^{(i)}|/|G^{(i+1)}| 
\end{equation}
for $ \,1 \leq i \leq l-1$ and
\begin{equation}\label{solvable3}
D^s((G^{(l)})^n ) \leq n \, |G^{(l)}|.
\end{equation}
Substituting the inequalities \eqref{solvable2} and \eqref{solvable3} in \eqref{solvable1}, we get 
$$D^s(G^n) \frac{1}{4} (4n )^{l} |G|,$$
which is the desired conclusion.
 \end{proof}

In 2006, Babai and Seress has been presented a relation between diameter and symmetric diameter of a finite group (See \cite[Corollary 2.2]{Babai:2006}). We apply this relation with the theorem \ref{solvable-symmetric} to find an upper bound for the diameter of $G^n$, where $G$ is a $p$-group. 
\begin{lem}\label{diameter-symmetric-diameter}
Let $G$ be a finite group and $X$ be a set of generators. The diameter and the symmetric diameter are related as follows:
$$\diam(G,X) \leq 2(\diam^s(G,X)+1)(|X|+1) \ln|G|.$$ 
\end{lem}
\begin{proof}
See \cite[Corollary 2.2]{Babai:2006}.
\end{proof}
\begin{theorem}\label{cor:solvable2}
Let $G$ be a solvable group of derived length $l$ and let $A$ be a generating set of $G^n$ of minimum size. Set $\rank(G)=\alpha, ~\rank(G/G')=\beta.$ The following inequality holds, 
 
$$D(G^n,A) \leq 2 (\frac{1}{4} (4n )^{l} |G|+1)(n \beta+1) n \ln |G|,$$
for $n \geq \alpha / \beta.$
In particular, if $G$ a $p$-group, then 
$$D(G^n) \leq 2 (\frac{1}{4} (4n )^{l} |G|+1)(n \beta+1) n \ln |G|,$$
for $n \geq 1.$

\end{theorem}
\begin{proof}
By Lemma \ref{diameter-symmetric-diameter} we have,
$$\diam(G^n,A) \leq 2 (\diam^s(G^n,A)+1)(|A|+1) n \ln |G|.$$ In addition, $\diam^s(G^n,A) \leq D^s(G^n)$ by definition. Now by using theorem \ref{solvable-symmetric} and Theorem \ref{rank-solvable-group} we get the desired conclusion. The second statement follows from these two facts: First, if $G$ is a $p$-group then every minimal generating set is a generating set of minimum size, which follows from the Burnside's Basis Theorem \cite{Hall:1976}. Second, by Corollary \ref{property}, if $G$ is a nilpotent group (note that every $p$-group is nilpotent) then $\rank(G)= \rank(G/G')$. 
\end{proof}

Now  we prove  a non symmetric version of  {\it Shereier Lemma} ( Lemma \ref{break-diameter}). This Lemma is essencial in the proof of our main theorem.   
\begin{lem}\label{diameter-Schreier-2}
Let $G$ be a finite group with a generating set $X$ and a normal subgroup $H$. Let $T$ be a right transversal of $G/H$ such that $$ Ml_{X}(T) \leq D(G/H),~ 1 \in T.$$ The following inequality holds:
$$\diam(G,X) \leq D(G/H) + (D(G/H) +1 + Ml_{X}(\{ t^{-1} \mid t \in T \})) D(H).$$ Furthermore, we have
\begin{align*}
D(G^n) \leq D(G^n/H^n)+ (1 + |G| D(G^n /H^n)) D(H^n).
\end{align*}
\end{lem}

\begin{proof}
Given $g \in G$, we have $ g = ht$ for some $h \in H$  and $ t \in T$. Hence $$ l_{X}(g ) \leq l_{X} (t)+ l_{X}(h).$$ Since $ Ml_{X}(T) \leq D(G/H)$, then $ l_{X}(g ) \leq D(G/H)+ l_{X}(h)$. Using Lemma \ref{Schreier} we get $l_{X}(h) \leq (D(G/H) +1 + Ml_{X}(\{ t^{-1} \mid t \in T \}) ) D(H)$. Combining these two facts gives the upper bound in the first inequality. Now we prove the second statement. Let $X'$ be a generating set of $G^n$ and let $T'$ be a right transversal of $G^n/H^n$ such that $$ Ml_{X'}(T') \leq D(G^n/H^n).$$ Proceeding as above for the case $n=1$, it suffices to show that $$Ml_{X'}(\{ t^{-1} \mid t \in T'\})\leq (|G|-1) D(G^n /H^n).$$ For given $t \in T'$ we have $$l_{X'}(t) \leq D(G^n/H^n).$$ Since $$ t^{-1}=t^{o(t)-1},$$ then we obtain $$l_{X'}(t^{-1}) \leq (o(t)-1)l_{X'}(t).$$ Hence, we have $$l_{X'}(t^{-1}) \leq (|G|-1)D(G^n/H^n),$$ since $$o(g) \leq |G|,$$ for every element $g \in G^n$. The proof is complete.      
\end{proof}


Now we are ready to prove our main theorem. 
\begin{theorem}\label{cor:solvable1}
Let $G$ be a non-abelian solvable group. Let 
$$\{1\}=G^{(l)} \triangleleft G^{(l-1)} \triangleleft \ldots \triangleleft G''\triangleleft  G' \triangleleft G$$ be the derived series of $G$.   
For $n\geq 2$, the following inequality holds:
$$D(G^n) \leq n^{l} |G| \prod_{i=0}^{l-2}( |G^{(i)}| +1).$$ 
\end{theorem}
 \begin{proof}
Since $(G^k)' = (G')^k$ for $k \geq 1$, then the derived series of $G^n$ is  
\begin{equation}\label{derived series}
\{1\}=(G^{(l)})^n \triangleleft (G^{(l-1)})^n \triangleleft \ldots \triangleleft (G'')^n \triangleleft  (G')^n\triangleleft G^n.
\end{equation}
 Applying Lemma \ref{diameter-Schreier-2} to the group $G^n$ with the subgroup $(G')^n$ gives 
\begin{align*} 
D(G^n) &\leq  D(G^n/ (G')^n) + (1 +|G|D(G^n/ (G') ^n)  D((G')^n)\\
&=D(G^n/ (G')^n)+D((G')^n)+|G|D(G^n/ (G') ^n)  D((G')^n)\\
& \leq D(G^n/ (G')^n)D((G')^n)+|G|D(G^n/ (G') ^n)  D((G')^n)\\
 & = D(G^n/ (G') ^n)  D((G')^n) (1 +|G|), 
\end{align*}
the second inequality follows from the fact that $D(G^n/ (G')^n),D((G')^n) > 1$ and this is because the quotient group $G/G'$ and the commutator subgroup $G'$ are not trivial.
By repeating  the process for the other subgroups in the series \eqref{derived series} we have 
\begin{equation}\label{diameter-G^n1}
D(G^n) \leq D(G^n/ (G') ^n)D((G')^n/ (G'') ^n) \ldots D((G^{(l-1)})^n) \prod_{i=0}^{l-2}( |G^{(i)}| +1).
\end{equation}l
Since for every group $G$ with a normal subgroup $H$ we have $G^n/H^n \cong (G/H)^n$, then 
\begin{equation}\label{diameter-G^n2}
D(G^n) \leq D((G/G') ^n)D((G'/G'') ^n) \ldots D((G^{(l-1)})^n) \prod_{i=0}^{l-2}( |G^{(i)}| +1).
\end{equation}
 All the quotient groups in the inequality \eqref{diameter-G^n2}  and the group $G^{(l-1)}$ are abelian.  On the other hand, for any abelian group $A$  we have  $$D(A^n) \leq n (|A| -\rank( A)) \leq n |A|$$ (see \cite[theorem 3.2 ]{Karimi:2017} ). Then we get
\begin{align*}
D(G^n)& \leq n^{l} D(G/ G')D(G'/ G'') \cdots\\
& D(G^{(l-2)} / G^{(l-1)}) D((G^{(l-1)}) \prod_{i=0}^{l-2}( |G^{(i)}| +1)\\
& \leq n^{l} |G/ G'| |G'/ G''| \cdots |G^{(l-2)} / G^{(l-1)}| |G^{(l-1)}|\prod_{i=0}^{l-2}( |G^{(i)}| +1)\\
& = n^{l} |G|\prod_{i=0}^{l-2}( |G^{(i)}| +1).\qedhere
\end{align*}
\end{proof}

 Note that the upper bound presented in theorem \ref{cor:solvable2} is just satisfied for $p$-groups. While, the upper bound presented in theorem \ref{cor:solvable1} not only is better,  but also is satisfied for all  non-abelian solvable groups.

  As an example of a non-abelian solvable group which is also a $2$-group we justify the upper bounds in theorems \ref{cor:solvable2} and \ref{cor:solvable1} for \emph{quaternion group} $Q_8$.
	Let $Q_8=\{\pm 1, \pm i, \pm j, \pm k \}$ be the quaternion group in which $$i^2=j^2=k^2=-1$$ and $$ij=k, jk=i, ki=j, ji=-k, kj=-i,ik=-j.$$
We have $Q_8' \cong Z_2$ and $Q_8 /Q_8' \cong Z_2 \times Z_2$. The length of the derived series of $Q_8$ is $2$. Hence, $l=2$ and $\beta = \rank(Z_2 \times Z_2)=2$ in the notations of theorems  \ref{cor:solvable2} and \ref{cor:solvable1}. Therefore by theorem \ref{cor:solvable2} we have 
$$D(Q_8^n) \leq 2n (32n^2+1)(2n+1)ln (8),$$ and by theorem \ref{cor:solvable1} we have   
$$D(Q_8^n) \leq 72n^2 .$$  

We now present a better upper bound for the diameter of the direct power of the quaternion group $Q_8$ by using Lemma \ref{diameter-Schreier-2} directly.  
\begin{ex}
For $n\geq1$ we have $D(Q_8^n) \leq 8n^2+3n$.
\end{ex}
\begin{proof} 
Consider the normal subgroup $H=\{1,-1\}$. Let $X$ be a generating set of $Q_8^n$. We have $H^n \triangleleft Q_8^n$. Let $T$ be a right transversal of $Q_8^n$ mod $H^n$ such that
$$1 \in T , Ml_X(T \setminus \{1\}) \leq D(Q_8^n /H^n).$$  
Using Lemma \ref{diameter-Schreier-2} we have
$$\diam(Q_8^n,X) \leq D(Q_8^n/H^n) + (D(Q_8^n/H^n) +1 + Ml_{X}(\{ t^{-1} \mid t \in T \} )) D(H^n).$$ 
On the other hand, since  $H \cong Z_2$ , $Q_8/H \cong Z_2 \times Z_2$, we have
\begin{equation}\label{quaternion}
\diam(Q_8^n,X) \leq 2n + (2n +1 + Ml_{X}(\{ t^{-1} \mid t \in T \})n.
\end{equation}
Since for every $g \in Q_8^n,~ g^4=1$, for every $t \in T,~ t^{-1}= t^3$. Hence, the following inequality holds:
$$l_{X}(t^{-1}) \leq 3 l_{X}(t) \leq 3  D(Q_8^n /H^n) \leq 6n.$$ 
Substituting $6n$ for $Ml_{X}(\{ t^{-1} | t \in T \}$ in \eqref{quaternion} we get
$$D(Q_8^n) \leq 8n^2+3n.\qedhere $$
\end{proof}

  
\section{Acknowledgments}
 The second author wishes to thank the University of Arak, for the invitation and hospitality. And the International Science and Technology Interactions (ISTI) for financial support.

\bibliographystyle{amsplain}
\bibliography{nasimref}

\end{document}